\documentclass[12pt]{amsart}
\usepackage{amsmath,amsthm,amssymb,setspace,tikz-cd,xcolor,soul,bm,wasysym,microtype}
\usepackage[capitalise]{cleveref}

\def\g{\gamma}

\def\io{\iota}

\def\sig{\sigma}
\def\vp{\varphi}

\def\t{\tau}

\def\ds{\displaystyle}

\def\cat{\mathcal C}
\def\s{\mathcal S}
\def\f{\mathcal F}

\renewcommand{\hom}{\operatorname{Hom}}
\DeclareMathOperator{\syl}{Syl}

\DeclareMathOperator{\inj}{Inj}

\DeclareMathOperator{\red}{red}
\DeclareMathOperator{\col}{col}
\DeclareMathOperator{\crit}{crit}

\renewcommand{\hat}{\widehat}
\renewcommand{\subset}{\subseteq}

\newtheorem{theo}{Theorem}
\newtheorem{lem}[theo]{Lemma}
\newtheorem{prop}[theo]{Proposition}
\theoremstyle{definition}
\newtheorem{defn}[theo]{Definition}

\numberwithin{theo}{section}
\makeatletter
\newtheorem{repeattheo@}{Theorem}

\makeatother

\usepackage[
textwidth=3cm,
textsize=small,
colorinlistoftodos]
{todonotes}

\begin{document}	
	\title[Contractibility after Steinberg]{Contractibility of the orbit space of a saturated fusion system after Steinberg}
	\author{Omar Dennaoui and Jonathon Villareal}
	\date{October 15, 2024}
	\begin{abstract}
		Recently, Steinberg used discrete Morse theory to give a new proof of a theorem of Symonds that the orbit space of the poset of nontrivial $p$-subgroups of a finite group is contractible. We extend Steinberg's argument in two ways, covering more general versions of the theorem that were already known. In particular, following a strategy of Libman, we give a discrete Morse theoretic argument for the contractibility of the orbit space of a saturated fusion system.
	\end{abstract}
	\maketitle

	\section{Introduction} 
	
	Let $G$ be a finite group and define $\s_p(G)$ to be the poset (under inclusion) of all non-trivial $p$-subgroups of $G$. The group $G$ acts naturally on $\s_p(G)$ by conjugation. The order complex $|\s_p(G)|$ has simplices the chains of inclusions in $\s_p(G)$ and retains the action of $G$ via conjugation since conjugation is inclusion preserving. This complex was first studied by Brown in \cite{brown1974euler} and \cite{brown1976high}. That the quotient space $|\s_p(G)|/G$ is contractible was first proven by Symonds \cite{symonds1998orbit} and an alternative proof was given later by Grodal in \cite{grodal2023endo}.
	
	\begin{theo}[\cite{symonds1998orbit}] \label{th:Symonds}
		Let $G$ be a finite group and let $\cat$ be a non-empty collection of nontrivial $p$-subgroups of $G$ that is closed under $G$-conjugacy and passage to overgroups that are $p$-groups. Then $N(\cat)/G$ is contractible.
	\end{theo}
	
	Linckelmann and Libman proved an extension of \Cref{th:Symonds} valid for a saturated fusion system $\f$ on a finite $p$-group $S$. In that context, $\cat$ is a closed $\f$-collection, a collection of subgroups of $S$ that is closed under $\f$-conjugacy and passage to overgroups in $S$ (see \Cref{def: clsd f col}), and $N(\cat)/\f$ is the quotient of the nerve $N(\cat)$ of the poset by the equivalence relation given by $\f$-conjugacy of chains.
		
	\begin{theo}[\cite{linckelmann2009orbit} and \cite{libman2008webb}] \label{th: fs contract}
		Let $ \f $ be a saturated fusion system over the finite $ p $-group $ S $ and let $\cat$ be a closed $\f$-collection. Then the space $ N(\cat)/\f $ is contractible.
	\end{theo} 
	
	Recently, Steinberg \cite{steinberg2023contractibility} proved that $|\s_p(G)|/G$ is contractible for any finite group $G$ using discrete Morse theory. This result is slightly weaker than \Cref{th:Symonds} (it handles the case $\cat = \s_p(G)$), but it gives a geometric reasoning for the contractibility of the space $|\s_p(G)|/G$. The goal of this paper is to extend Steinberg's argument to general collections $\cat$ and to saturated fusion systems. \par
	
	To prove \Cref{th: fs contract} using discrete Morse theory, we extend \cref{th:Symonds} to the case of a pseudo finite group $G$ (\cref{df:Lib pseudo}) by constructing a Morse matching with a single critical cell for $N(\cat)/G$. Our argument differs from Steinberg's in that we do not pass to normal chains of $p$-subgroups. This allows us to extend to arbitrary collections, recovering \cref{th:Symonds} in the original generality, and also to finish the proof of \cref{th: fs contract} following a strategy of Libman.
	
	\subsection*{Notation and Terminology}
	We fix the following notation and terminology. When $ G $ is a group, $ P\leq G $, and $ g\in G $, we write $ P^g=g^{-1}Pg $. We also write the $ G $-conjugacy class of $ P $ as follows:        
	\begin{align*}
		P^G=\{Q\leq G\mid\text{ there is }g\in G\text{ such that }P^g=Q\}
	\end{align*}
	When $ g\in G $, we write $ c_g $ for the right-handed conjugation homomorphism $ x\mapsto g^{-1}xg $ and its restrictions. We write 
	\begin{align*}
		\hom_G(P,Q)=\{c_g\mid g\in G,\text{ }P^g\leq Q\}.
	\end{align*}
	Given a prime $ p $, we denote by $\syl_p(G)$ the set of all Sylow $p$-subgroups of $G$ (when such exist; see \cref{contractibility}).
	
	\section{Discrete Morse theory}\label{discrete morse theory}
		
	The general idea of discrete Morse theory is that we would like to strip away all non-essential data from a simplicial set and only keep a record of the nondegenerate simplices which, if removed, would alter the homotopy type.	
	
	\begin{defn} \label{def: Morse match}
		Let $X$ be a simplicial set. Suppose given  
		\begin{enumerate}
			\item a partition of the nondegenerate simplices (or cells) into \textit{critical} cells, \textit{redundant} cells, and \textit{collapsible} cells. 
			\item a bijection $c$ from the set of redundant cells to the set of collapsible cells such that, for each redundant $n$-cell $\t$, $c(\t)$ is an $(n+1)$-cell and $\tau$ is a face of $c(\tau)$, and
			\item a choice of index $\io(\t)\in \{0,1,\dots, n + 1\}$ such that $\t = d_{\io(\t)}(c(\t))$, where $d_i$ denotes the $i$-th face map.
		\end{enumerate}  
		We say that these data are a \textit{Morse matching} (or \textit{collapsing scheme} in \cite{brown1992geometry}) for $X$ if the following condition is satisfied: if $D$ is the digraph whose vertex set consists of the redundant and collapsible cells and whose edges are of the form $\t \rightarrow c(\t)$ for $\t$ redundant, and $\sig \rightarrow d_j(\sig)$, where $j\neq \io(c^{-1}(\sig))$, for $\sig$ collapsible, then there is no infinite directed path in $D$.
	\end{defn}
	
	Note that we ignore degenerate simplices since we are interested in the geometric realization of $X$, in which they are identified with lower dimensional simplices. In \Cref{def: Morse match}, \textit{critical} cells are the ones that alter the homotopy type if removed. By \Cref{th:Morse collapse} below, called the fundamental theorem of discrete Morse theory, they are the only ones that matter with respect to homotopy. \textit{Redundant} and \textit{collapsible} cells come in pairs because of the bijection $c$, and are the ones that will collapse without altering the homotopy type.
	    
	\begin{theo}[\cite{brown1992geometry}, \cite{forman1998morse}] \label{th:Morse collapse}
		Let $X$ be a simplicial set. Given a Morse matching for $X$, there is a CW-complex $Y$ and a quotient map $q : |X| \rightarrow Y$ of CW-complexes such that  
		\begin{enumerate}
			\item for each $n$, the $n$-cells of $Y$ are in bijection with the critical $n$-cells in $X$ (with characteristic maps the compositions of $q$ with the characteristic maps to $|X|$), and
			\item $q$ is a homotopy equivalence. 
		\end{enumerate} 
	\end{theo}
	
	Our goal is to apply \cref{th:Morse collapse} by constructing a Morse matching on a closed collection of finite $p$-subgroups of a pseudo finite group at the prime $p$.
	
	\section{Contractibility in pseudo finite groups}\label{contractibility}
	
	Let $ G $ be a (possibly infinite) group.  A finite $p$-subgroup $S$ of $G$ is a \emph{Sylow $p$-subgroup} of $G$ if every finite $p$-subgroup of $G$ is conjugate to a subgroup of $S$. By Sylow's Theorem, this is equivalent to the usual definition of Sylow $ p $-subgroup in case $G$ is finite. If $G$ is a group and $\g = (P_0\leq P_1\leq\cdots\leq P_n)$ is a chain of subgroups, the normalizer of $\g$ in $G$ is the intersection 
	\begin{align*}
		N_G(\g)=\bigcap_{i=0}^nN_G(P_i)
	\end{align*} 
	of the normalizers of the members of the chain. The following definition is due to Libman \cite[Definition 3.7]{libman2008webb}.
	 	
	\begin{defn}\label{df:Lib pseudo}
		A group $G$ is \textit{pseudo finite at $p$} if for every chain $\g$ of finite $p$-subgroups of $G$, the normalizer $N_G(\g)$ has a Sylow $p$-subgroup.
	\end{defn}	
	
	Here, we allow the trivial subgroup to appear in the chain $\g$. In particular, a pseudo finite group has a Sylow $p$-subgroup. The following is the main theorem of the paper. It generalizes Steinberg in two ways; the Morse matching is applicable to pseudo finite groups $G$ at the prime $p$ and to arbitrary \textit{closed} sub-collections of $\s_p(G)$.  

	\begin{theo} \label{th: |S|/G contract}
			Let $ G $ be a pseudo finite group at the prime $ p $ and $ \cat\subset\s_p(G) $ be a collection of non-trivial $ p $-subgroups of $ G $ that is closed under conjugation and passage to overgroups. The simplicial set $N(\cat)/G$ admits a Morse matching in which the set of critical cells consists of a single cell of dimension 0. Hence, $N(\cat)/G$ is contractible.  
	\end{theo}
	
	Note that the group $ G $ has an order-preserving action via conjugation on $ \cat $. This action extends to a simplicial action on $ N(\cat) $. \par 
	
	Given an $n$-simplex $\g = (P_0,\ldots,P_n)$ in $N(\cat)$, we write 
	\begin{align*}
		[\g] = [P_0,\ldots,P_n]
	\end{align*}
	for the orbit of $\g$ under this action. We will prove \Cref{th: |S|/G contract} in three steps. First, we define the collections of redundant, collapsible, and critical cells and show that they partition the nondegenerate cells of $N(\cat)/G$. Second, we define the maps $c$ and $\io$. Lastly, we show that there are no infinite directed paths in the associated digraph.
	
	
	\subsection{The partition} 
	Recall that we need to partition the nondegenerate cells of $N(\cat)/G$ into critical cells, redundant cells, and collapsible cells. Let $\Gamma(\cat)\subset N(\cat)$ be the set of nondegenerate simplices $\g=(P_0,\ldots,P_n)$. By convention, $P_{-1}=1$ is the trivial subgroup. \par 
	
	We say that $\g$ is \textit{redundant} if for any $Q\in\syl_p(N_G(\g))$ and any $0\leq i\leq n$, we have $P_i\neq QP_{i-1}$.  If $\gamma$ is not redundant , we say $\g$ is \textit{collapsible} if $n\geq1$ and \textit{critical} if $n=0$. \par 
	
	Now let $\red(\cat), \col(\cat), $ and $\crit(\cat)$ be the collections of redundant, collapsible, and critical cells, respectively. By construction, these sets form a partition of $\Gamma(\cat)$, and they are invariant under conjugation because $N_G(\g)^g=N_G(\g^g)$ for any element $g\in G$. This partition gives us a partition on $\Gamma(\cat)/G$ consisting of $\red(\cat)/G$, $\col(\cat)/G$, and $\crit(\cat)/G$. \par 
	
	The next lemma shows that there is only one critical cell and it is precisely the orbit of the Sylow $p$-subgroups of $G$. 
	\begin{lem}\label{lem: crit equal sylow}
		If $[\gamma]$ is a nondegenerate cell, then $[\gamma]$ is a critical cell if and only if $[\gamma]=[S]$ where $S\in\syl_p(G)$. 
	\end{lem}
	\begin{proof}
		Assume that $[\gamma]=[S]$. Since $S=S\cdot1$ and $S\in\syl_p(N_G(S))$, it is critical. \par 
		
		Conversely, suppose that $[\gamma]$  is a critical cell. There is some $P_0$ such that $[\gamma]=[P_0]$ and $P_0\in\syl_p(N_G(P_0))$. In particular, $P_0\in\syl_p(N_S(P_0))$. Therefore, there is some $g\in N_S(P_0)\leq G$ such that $P_0^g=S$. This gives us that $[\g]=[P_0]=[S]$.
	\end{proof}
	
	
	\subsection{The maps $\bm{c}$ and $\bm{\io}$}
	Let $\t=(P_0,\ldots,P_n)$ be any redundant cell. We define $\io([\t])=i+1$ where $0\leq i\leq n$ is maximal such that $Q\nleq P_i$ for each $Q\in\syl_p(N_G(\t))$. \par 
	
	We now define $c:\red(\cat)/G\to\col(\cat)/G$ by, for each $\t\in\red(\cat)$, fixing $Q\in\syl_p(N_G(\t))$ and setting $c([\t])=c_Q([\t])$, where  
	\begin{align*} \label{def c} 
		c_Q([\t])=    
		\begin{cases}
			[P_0, \ldots , P_{\io([\t])-1}, QP_{\io([\t])-1}, P_{\io([\t])}, \ldots ,P_n] & \text{ if $1 \leq \io([\t]) \leq n$}\\
			[P_0, \ldots ,P_n, QP_n] &  \text{ if $\io([\t]) = n+1$}.
		\end{cases}   
	\end{align*} \par 
	
	We need to show that $c$ is a well-defined bijection between redundant and collapsible cells in $\Gamma(\cat)/G$. In the proof, we will write $c_{Q}(\tau)$ for the representative of the orbit $c_Q([\tau])$. \par
	
	\smallskip\noindent
	\textbf{$\bm{c}$ and $\bm{\iota}$ are independent of the choice of Sylow subgroup:}
	Let $\t = (P_0, P_1, \ldots ,P_n)\in\red(\cat)$ and suppose $Q,R\in\syl_p(N_G(\t))$. There exists $g\in N_G(\t)$ such that $Q^g=R$. Notice that for
	all $0\leq j\leq n$, $P_j^g=P_j$, and since conjugation is inclusion preserving, there is a unique $i$ such that $Q,R\not\leq P_i$ and $Q,R\leq P_{i+1}$ unless $i=n$. We then have $P_i<QP_i<P_{i+1}$ or $P_n<QP_n$ and    
	\begin{align*}
		c_Q(\t)^g &= (P_0, P_1, \ldots ,QP_{i},\ldots, P_n)^g\\
		&=(P_0^g, P_1^g, \ldots ,(QP_{i})^g,\ldots, P_n^g) \\
		&=(P_0, P_1, \ldots ,Q^gP_{i}^g,\ldots, P_n) \\
		&=(P_0, P_1, \ldots ,RP_{i},\ldots, P_n)\\
		&= c_R(\t),
	\end{align*}
	for example. That is, $[c_Q(\tau)] = [c_R(\tau)]$ and $\iota([c_Q(\t)]) = i+1 = \iota([c_R(\t)])$.
		 
	\smallskip\noindent
	\textbf{$\bm{c}$ and $\bm{\iota}$ are well-defined.} Let $[\t],[\t']\in\red(\cat)/G$ be such that $[\t]=[\t']$. That is, there is some $g\in G$ such that $\t^g=\t'$. This gives us that $N_G(\t)^g=N_G(\t')$, and $Q$ is a Sylow $p$-subgroup of $N_G(\t)$ if and only if $Q^g$ is a Sylow $p$-subgroup of $N_G(\t')$. Since conjugation is inclusion preserving, it follows from the independence of the choice of Sylow $p$-subgroup that $c([\t]) = [c_Q(\t)] = [c_{Q^g}(\t')] = c([\t'])$. \par
	
	\smallskip \noindent \textbf{$\bm{c([\tau])}$ is the orbit of a nondegenerate simplex.} Note that since $P_0$ is a normal $p$-subgroup of $N_{G}(\tau)$ and $Q$ is a Sylow subgroup, $P_0 \leq Q$.  By definition of redundant, this must be a proper inclusion. That is, $Q \nleq P_0$, so $\io([\t])\neq0$. Also note that $QP_{\io([\t])-1} < P_{\io([\t])}$ is a proper inclusion, since $[\tau]$ is redundant.  \par
	
	\smallskip\noindent
	$\bm{c}$\textbf{ maps redundant to collapsible.} Next, we show that $c([\t])$ is collapsible. Notice that if $Q$ is a Sylow $p$-subgroup of $N_G(\tau)$, then $Q\leq N_G(c_Q(\t))\leq N_G(\t)$, so $Q$ must also be a Sylow $p$-subgroup of $N_G(c_Q(\t))$. Since $P_{\io([\t])-1}$ and $QP_{\io([t])-1}$ are consecutive members of the chain $c_Q(\t)$, this gives us $c([\t]) = [c_Q(\t)]$ is collapsible. \par
	
	\smallskip\noindent
	\textbf{$\bm{c}$ is injective.} Let $[\t]=[P_0,\ldots,P_n]$ and $[\t']=[P_0',\ldots,P_n']$ such that $c([\t])=c([\t'])$. Let $Q\in\syl_p(N_G(\t))$ and $Q'\in\syl_p(N_G(\t'))$. 
	We have the following:
	\begin{align*}
		c([\t])&=[P_0,\ldots,P_{\io([\t])-1},QP_{\io([\t])-1},P_{\io([\t])},\ldots,P_n]=[\sig] \\
		c([\t'])&=[P'_0,\ldots,P'_{\io([\t'])-1},Q'P'_{\io([\t'])-1},P'_{\io([\t'])},\ldots,P'_n]=[\sig']
	\end{align*}
	Since $c([\t])=c([\t'])$, there is some $g\in G$ such that $\sig^g=\sig'$. Also, since $Q\in\syl_p(N_G(\sig))$, we have that $Q^g\in\syl_p(N_G(\sig'))$. Hence, there is some $h\in N_G(\t')$ such that $Q^{gh}=Q'$. Therefore, by the definition of $\io([\t])$, we have that $\io([\tau])$ is the unique maximal index such that $Q\not\leq P_{\io([\tau])-1}$, and thus
	\begin{align*}
		Q^g&\not\leq P^g_{\io([\tau])-1} & &\text{Conjugation preserves ordering} \\
		Q^{gh}&\not\leq P^g_{\io([\tau])-1} & &\text{Conjugation by }h\text{ does not affect }P^g_{\io([\tau])-1} \\
		Q'&\not\leq P'_{\io([\tau])-1}
	\end{align*}
	By the maximality of $\io([\tau'])$, $\io([\tau]) \leq \io([\tau'])$. Using a symmetric argument, we obtain that $\io([\t'])\leq\io([\t])$ giving the equality $\io([\t])=\io([\t'])$. Therefore, $[\tau]= d_{\iota([\tau])}(c([\tau])) = d_{\iota([\tau'])}(c([\tau'])) = [\tau']$. \par
	
	\smallskip\noindent
	\textbf{$\bm{c}$ is surjective.} Let $[\sig]=[R_0,\dots,R_{n+1}]$ be a collapsible $(n+1)$-cell, which means that $QR_{i-1} = R_{i}$ for some $1\leq i\leq n+1$, where $Q\in\syl_p(N_G(\sig))$. Suppose towards a contradiction that $Q\notin\syl_p(N_G(\t))$ where $[\t]=d_{i}([\sig])$. We can choose $R\in\syl_p(N_G(\t))$ such that $Q<R$. Since $R$ is a finite $p$-group, it follows that 
	\begin{align*}
		Q&<N_R(Q) \\ 
		&\leq R\cap N_G(Q) \\ 
		&\leq N_G(\t)\cap N_G(Q) \\
		&\leq N_G(\sig)
	\end{align*}
	This is a contradiction as $N_R(Q) \leq R$ and $Q$ is a Sylow $p$-subgroup of $N_G(\sig)$. Therefore, $Q$ is indeed a Sylow $p$-subgroup of $N_G(\t)$, and we get that $c([\t])=[\sig]$. This shows that $c$ is surjective. \hfill 
	
	\subsection{No infinite directed paths}
	Let $D$ be the associated digraph whose vertex set consists of the union of $\red(\cat)/G$ and $\col(\cat)/G$ and whose edges are of the form $[\t] \rightarrow c([\t])$ for $[\t]\in\red(\cat)/G$, and $[\sig] \rightarrow d_j([\sig])$, where $j\neq \io(c^{-1}([\sig]))$, for $[\sig]\in\col(\cat)$.\par 
	
	Define for any $\g\in N(\cat)$
	\begin{align*}
		h([\g])=\log_p|Q|
	\end{align*}
	where $Q\in\syl_p(N_G(\g))$. We say that $h([\g])$ is the \textit{height} of $\g$. If two $n$-cells are equal in $N(\cat)/G$, then they are conjugate and their Sylow $p$-subgroups are conjugate. Hence, $h$ is well defined. The function $ h $ is also bounded above by $\log_p|S|$.\par 
	
	We need to show that there is no infinite directed path in $D$. The following proposition explains how the height of cells change over edges in $D$.
	
	\begin{prop}\label{prop:ht is nondec}
		Suppose that there is an edge in $D$ from $[\sig]$ to $[\g]$. If $[\sig]$ is collapsible and $[\gamma]$ is redundant, then $h([\g]) > h([\sig])$. Otherwise, $h([\g]) = h([\sig])$.
	\end{prop}
	\begin{proof}
		Let $[\sig]$ be a collapsible $n$-cell and $Q\in\syl_p(N_G(\sig))$. By definition of $D$, if there is an edge from $[\sig]$ to $[\gamma]$, then $[\gamma]=d_j([\sigma])$ for some $j\neq\io(c^{-1}([\sigma]))$. As $Q$ normalize $\sig$, it normalizes $d_j(\sigma)$. Since $Q$ is a finite $p$-group, we can choose $R\in\syl_p(N_G(d_j(\sig)))$ such that $R\geq Q$. This gives us that $h(d_j([\sig]))\geq h([\sig])$. \par 
		
		Notice that by our choice of $R$, $h(d_j([\sig]))=h([\sig])$ if and only if $Q=R$. It follows from the definition of collapsible that equality holds if and only $d_j([\sig])$ is collapsible. This finishes the proof of the proposition for edges that begin at a collapsible cell. \par 
		
		Now suppose that $[\sig]$ is a redundant cell. By definition of $D$, we have that $[\g]=[c_Q(\sig)]$ for a fixed $Q\in\syl_p(N_G(\sig))$. Since $Q$ is also a Sylow subgroup of $N_G(c_Q(\sig)))$, it follows that $h([\gamma]) = h([\sigma])$. 
	\end{proof}
	
	We now show that there is no infinite directed paths in $D$. 
	\begin{proof}
		Fix a directed path in $D$. By definition of $D$, there is no edge connecting two redundant cells. That is, each edge in the path goes from a collapsible cell to a redundant cell, a collapsible cell to a collapsible cell, or a redundant cell to a collapsible cell. \par  
		
		Since the height of any cell in the path is bounded above by $h([S])$ where $S\in\syl_p(G)$, there are only finitely many edges in the path going from a collapsible cell to a redundant cell by \Cref{prop:ht is nondec}. \par 
		
		Over each edge in the path going from a collapsible cell to a collapsible cell the dimension of the cell decreases. Since dimension is bounded, there are finitely many edges in a string of collapsible cells before the path reaches a redundant cell (or terminates). Since there are only finitely many edges going from a collapsible cell to a redundant cell by above, there are only finitely many edges in the path going from a collapsible cell to a collapsible cell. \par
		
		Finally, each edge from a redundant cell to a collapsible cell is immediately preceded by an edge from a collapsible to the redundant (or the redundant cell started the path) and we have seen there are only finitely many of these. Hence, there are a finitely many of edges going from a redundant to a collapsible. \par 
		
		Therefore, the path is finite. 
	\end{proof}
	\noindent 
	This finishes the proof that the given data define a Morse matching for $N(\cat)/G$ with a single critical cell. \Cref{th:Morse collapse} implies that $N(\cat)/G$ is contractible and completes the proof of \Cref{th: |S|/G contract}. Note that \Cref{th: |S|/G contract} implies \Cref{th:Symonds} as all finite groups are pseudo finite at a given prime $p$ that divides $|G|$.
	
	\section{Application to Saturated Fusion Systems}
	When $ G $ is a group (not necessarily finite) and $ S $ is a finite $p$-subgroup of $G$, we can form the \textit{fusion system}, denoted $ \f_S(G) $, as follows:
	\begin{itemize}
		\item The objects of $ \f_S(G) $ are the subgroups of $ S $. 
		\item Given any two subgroups $ P,Q\leq S $, $ \hom_\f(P,Q)=\hom_G(P,Q) $. 
	\end{itemize}
	\noindent  
	Broto, Levi, and Oliver defined fusion systems in \cite{broto2003homotopy}, in part based on some earlier ideas of Puig. Modified but equivalent definitions of abstract fusion systems can be found in \cite{aschbacher2011fusion}.
	\begin{defn}[\cite{puig2006frobenius}, \cite{aschbacher2011fusion}]\label{def:abs fs}
		A \textit{fusion system} over a finite $ p $-group $ S $ is a category $ \f $ such that 
		\begin{itemize}
			\item The objects of $ \f $ are the subgroups of $ S $.
			\item Given any two objects $ P,Q\leq S $, the morphism set is such that $ \hom_S(P,Q)\subset\hom_\f(P,Q)\subset\inj(P,Q) $, where $\inj(P,Q)$ is the set of all injective group homomorphisms from $P$ to $Q$. 
			\item Each $ \vp\in\hom_\f(P,Q) $ is the composite of an isomorphism in $ \f $ followed by an inclusion. 
		\end{itemize}
		We say that $ P $ and $ Q $ are $ \f $\textit{-conjugate} if they are isomorphic as objects in $ \f $. Write $ P^\f $ to denote the collection of all subgroups of $ S $ that are $ \f $-conjugate to $ P $.  
	\end{defn} 
		
	A fusion system $\f$ is said to be \emph{saturated} if it satisfies further axiom(s) as detailed in Definition 2.2 in \cite{aschbacher2011fusion}. When $G$ is a finite group and $P\in\syl_p(G)$, then $\f_P(G)$ is a saturated fusion system. However, a fusion system of the form $\f_P(G)$ is usually not saturated if $P\notin\syl_p(G)$. \par 
	
	A saturated fusion system $ \f $ is said to be realizable by the (possibly infinite) group $G$ if there is some $ S\in\syl_p(G) $ such that $ \f=\f_S(G) $.  There are examples of saturated fusion systems that are not realized by any \emph{finite} group in this way. However, following realizability result was shown independently by Robinson and by Leary and Stancu.
	\begin{theo}[{\cite[Theorem 2]{robinson2007amalgams}, \cite[Theorem 2]{leary2007realising}}]\label{th:rob}
		Let $ \f $ be a saturated fusion system over a finite $ p $-group $ S $. There is some (possibly infinite) group $ G $ and $S \in \syl_p(G)$ such that $ \f=\f_S(G) $.  
	\end{theo}
	
	In his paper, Libman was able to show that if a (possibly infinite) group $ G $ realizes a saturated fusion system $ \f $ over the finite $ p $-group $ S $, then $ G $ is pseudo finite at the prime $ p $. We state his results below. 
	\begin{prop}[\cite{libman2008webb}, Proposition 3.9] \label{th:Lib sat fs realize pseudo}
		A group which realizes a saturated fusion system is pseudo finite at the prime $ p $. 
	\end{prop}
	
	For the remainder of this paper, fix a saturated fusion system $\f$ over a fixed finite $p$-group $S$. By \Cref{th:rob}, we can fix a group $G$ such that $\f=\f_S(G)$ and $S\in\syl_p(G)$, and using \Cref{th:Lib sat fs realize pseudo}, we know that $G$ is a pseudo finite group at the prime $p$. 
	\begin{defn}[\cite{libman2008webb}, Definition 1.2]\label{def: clsd f col}
		Fix a fusion system $\f$ of $S$. An $\f$\textit{-collection} is a union of $\f$-conjugacy classes of subgroups of $S$. An $\f$-collection $\cat$ is \textit{closed} if a subgroup $Q\leq S$ belongs to $\cat$ whenever it contains an element of $\cat$. 
	\end{defn} 
	
	We now let $ \cat $ be any nonempty closed $\f$-collection. It is clear that $\cat$ is a poset under inclusion. We again denote the nerve of $\cat$ as $N(\cat)$. On $N(\cat)$, define an equivalence relation $\sim_{\f}$ as follows: Given two $ n $-simplices, we have $$(P_0,P_1,\ldots,P_n)\sim_{\f}(Q_0,Q_1,\ldots,Q_n)$$ if and only if there is some isomorphism $ \vp\in\hom_\f(P_n,Q_n) $ such that $ \vp(P_i)=Q_i $ for every $ i\in\{0,1,2,\ldots,n\} $. We will denote the quotient $ N(\cat)/\!\!\sim_{\f} $ as $ N(\cat)/\f $. \par
	
	Define 
	\begin{align*}
		\hat{\cat}=\ds\bigcup_{P\in\cat}P^G
	\end{align*}
	where $P^G$ denotes the $G$-conjugacy class of $P$. By construction, the nerve $N(\hat{\cat})$ is closed under $G$-conjugation, so we can form the quotient $N(\hat{\cat})/G$. Using the following result from Libman, we have that $ N(\cat)/\f $ is isomorphic to $ N(\hat{\cat})/G $ as simplicial sets. \clearpage
	\begin{prop}[\cite{libman2008webb}, Proposition 3.10] \label{th:Lib s htpy hats}
		Let $ G $ be a group that realizes the saturated fusion system $ \f $ on the finite $ p $-group $ S $. Let $ \cat $ be an $ \f $-collection. Then $ N(\cat)/\f\to N(\hat{\cat})/G $ is an isomorphism of simplicial sets.  
	\end{prop}
	With the work in \Cref{contractibility}, we can prove \Cref{th: fs contract} using discrete Morse theory:    
	\begin{theo}
		Let $\f$ be a saturated fusion system over the finite $ p $-group $ S $ and let $\cat$ be a closed $\f$-collection. The simplicial set $N(\cat)/\f$ admits a Morse matching in which the set of critical cells consists of one cell of dimension 0, and hence $N(\cat)/\f$ is contractible.
	\end{theo}
	\begin{proof}
		By \Cref{th:rob} and \Cref{th:Lib sat fs realize pseudo}, $ \f=\f_S(G) $ for some pseudo finite group $ G $ at the prime $ p $ and $ S\in\syl_p(G) $. By \Cref{th:Lib s htpy hats}, we know that $ N(\cat)/\f $ is isomorphic to $ N(\hat{\cat})/G $ as simplicial sets. By \Cref{th: |S|/G contract}, we know that $ N(\hat{\cat})/G $ admits a Morse matching with in which the set of critical cells consists of one cell of dimension~0. Applying \Cref{th:Morse collapse}, $N(\cat)/\f$ is contractible.
	\end{proof}
	
	\section*{Acknowledgments}
	We would like to express our gratitude to our advisor, Justin Lynd, for his guidance and support throughout the writing of this paper. His dedication to our work, keen insights, and meticulous review were instrumental in shaping the content of this manuscript. This paper would not have been possible without his mentorship. We would like to extend our appreciation to Philip Hackney for his insight on the topological aspects of this paper. We are grateful for his time, insightful comments, and dedication to ensuring the integrity of our mathematical reasoning. We also want to thank Jesper Grodal and Benjamin Steinberg for reaching out after making our work public with insights and recommendations on a previous version of this paper. We finally want to thank the referee for pointing out a mistake in the construction of the Morse matching and taking the time to write a detailed report on the previous version of this paper.
	
	\bibliographystyle{alpha}
	\bibliography{ref.bib}
\end{document}